\documentclass[12pt]{article}
\usepackage{amsfonts}
\usepackage{amsfonts}  
\usepackage{amssymb}
\usepackage{amsthm}
\usepackage{amsmath}
\usepackage{color}

\title{Lipschitz interior  regularity for the  viscosity  and weak solutions of the Pseudo $p$-Laplacian Equation. } 

\author{F. Demengel}

\date{Laboratoire AGM,  UMR 8088, University   of   Cergy Pontoise, France}

\catcode`@=11 \@addtoreset{equation}{section} \catcode`@=12

\newtheorem{theo}{Theorem}[section]
\newtheorem{prop}[theo]{Proposition}
\newtheorem{rema}[theo]{Remark}
\newtheorem{defi}[theo]{Definition}

\newtheorem{lemme}[theo]{Lemma}

\def\R{\mathbb  R}
\def\grad{\nabla}

\begin{document}

\maketitle
\begin{abstract}
We  consider the pseudo-$p$-Laplacian operator: 
 $\tilde \Delta_p u = \sum_{i=1}^N \partial_i (|\partial_i u|^{p-2}  \partial_i u)=(p-1) \sum_{i=1}^N |\partial_i u|^{p-2} \partial_{ii} u$ for $p>2$. 
 We prove interior regularity results  for the  viscosity (resp. weak) solutions  in the unit ball 
  $B_1$ 
  of $\tilde \Delta_p u =(p-1) f$   for  $f\in {\cal C} (\overline{B_1})$ (resp. $f\in L^\infty(B_1)$) : Firstly the H\"older local  regularity 
   for any   exponent $\gamma <1$, 
  recovering  in that way  a    known result  about weak solutions.   
   In a second  time we prove the Lipschitz  local  regularity. 
    \end{abstract}
    \section {Introduction}

     This paper is devoted to the local Lipschitz regularity for viscosity solutions of the equation 
     \begin{equation}\label{eq1}
     \sum_{i=1}^N \partial_i (|\partial_i u|^{p-2} \partial_i u) =(p-1) f
      \end{equation}
       The operator on the left hand side is known as the pseudo-$p$-Laplace operator, and the equation above is the Euler Lagrange equation associated to the energy functional 
    $$
        {1\over p}\int \sum_{i=1}^N |\partial_i u|^p +(p-1)\int fu. 
  $$

          Even if this equation seems very similar to the usual $p$-Laplace equation (existence of solutions, comparison theorems),  the usual methods to prove regularity results   cannot easily  be adapted here. This is  mainly due to the fact that the operator degenerates on non bounded sets in $\R^N$.   
          
          \bigskip
            
         Several directions have been taken by different  people, and except for Belloni and Kawohl in \cite{BK} , all the solutions they   studied were  weak solutions, i.e.  such that  $u\in  W^{1,p}_{loc}$ and the equation ( \ref{eq1}) is intended in the distribution sense.

          The first  regularity results for this type of equations may be found  in the pioneer  paper of  Uraltseva and Urdaletova, \cite{UU} :  Among other results they prove    that,  for $f$ replaced by $f(x,u) $ with some specific conditions of growth   with respect to  $u$,  and for $p>3$ , then the solutions are Lipschitz continuous. 
          The Lipschitz regularity in the case $p<2$, with a right hand side $f\in L^\infty$  can be derived from the techniques used  by Fonseca,  Fusco and Marcellini in \cite{FFM} :  
                   One must point out that   when $p<2$   the notion of viscosity solution cannot   easily be defined,  since the operator  has no meaning on  points $\bar x$ for which  some test function $\varphi$  satisfies $\partial_i \varphi(\bar x) = 0$. Therefore  it is not  immediate to    obtain Lipschitz regularity for   the solutions using viscosity   arguments as here.

           In \cite{BB} the authors studied (among other things  ) the Lipschitz regularity for  equations  as  
            $$\left\{ \begin{array}{lc}
            -\mu \sum_i \partial_i( (|\nabla u|^{p-1}-\delta)^+{\partial_i  u\over |\nabla u|} ) -\sum_i \partial_i (\partial_i u|^{p_i-2} \partial_i u) = f& {\rm in} \ \Omega\\
            u = \varphi \ {\rm on } \  \partial \Omega
            \end{array}\right.$$
           Here the $p_i $ are $>1$, $p>1$ and $\mu >0$.  They prove a Lipschitz regularity result under some bounded slop condition on $\varphi$. The  case where $\mu = 0$,  even when all the $p_i$ are equal to each other  is not covered by  their proofs.

             -The degenerate case $p\geq 2$  has been  much explored,  and almost all  the techniques involved  are   variational : In particular  the regularity  results are obtained using Moser's iteration method.

                          Among the  recent regularity results obtained,   let us cite the paper of  Brasco and Carlier \cite{BC}, which
                          proves that for the widely   degenerate anisotropic equation, arising in congested optimal transport 
                           
               \begin{equation}\label{eq3}
              \sum_{i=1} ^N \partial_i\left( (|\partial_i u|-\delta_i)_+^{p-1} {\partial_i u\over |\partial_i u|} \right)= (p-1)f.
               \end{equation} (here  the $\delta_i$ are non negative numbers,  and $f\in L^\infty_{loc}$),  
                the solutions are in $W^{1, q}_{loc} $ for any $q < \infty$. In particular it implies the H\"older's regularity  of the solutions for any exponent $\gamma<1$, by means of the Sobolev Morrey's embedding.

                In \cite{BK} Belloni and Kawohl are  interested in the first eigenvalue for the Dirichlet problem of the pseudo $p$-Laplace operator : 
                They prove  existence and  uniqueness  of the first eigenfunction, up to a multiplicative constant.

                                     \bigskip
                     
                      The most  complete and strongest results about regularity  concerns the  widely   degenerate anisotropic equation    in  (\ref{eq3}) :  In \cite{BBJ},  Bousquet, Brasco and Julin 
                      prove  the following  Lipschitz regularity result  :

                      {\em If $N=2 $,   for any $p\geq 2$,  and for $f\in W^{1,p^\prime}_{loc} $, $({1\over p}+ {1\over p^\prime}=1$), or if $N\geq 3$,  $p\geq 4$, and $f\in W^{1, \infty}_{loc}$,  then  every weak solution   of (\ref{eq3}) is   locally Lipschitz continuous.}

                          In particular their results include the case where $\delta_i=0$ for all $i$,  under the regularity assumption on  $f$ above.  Once more the techniques  involved are variational.                           
                          \bigskip
                          
                       In the present paper we consider ${\cal C}$-viscosity solutions of  (\ref{eq1}) : 
                        In fact the result for viscosity solutions will be a corollary of  the  stronger result  
                        \begin{theo}\label{th1bis}
                        For any $p>2$, for any $f, g\in {\cal C} (\overline{B_1})$ and for any $u$, USC  and $v$ LSC which  satisfy  in the viscosity sense 
                        $$ \sum_i |\partial_i u|^{p-2} \partial_{ii} u \geq f \  {\rm and } \  \sum_i |\partial_i v|^{p-2} \partial_{ii} v  \leq   g  $$
                          then for any $r<1$, there exists $c_r$ such that  for any $(x,y) \in B_r^2$
                          $$ u(x)-v(y) \leq \sup (u-v) + c_r |x-y|.$$
                          \end{theo}
                           In particular the Lipschitz regularity result holds  true  for $u$ if $u$ is  both a sub and a super-solution of the equation,  even with some right hand sides differents. 
               Furthermore  the  Lipschitz continuity for solutions of (\ref{eq1}) requires only that  the right hand side    be    continuous.

                  The second important advantage of  the methods here employed is that they   can be applied  to study regularity of Fully Non Linear Operators on the model of the pseudo $p$-Laplace operator, but not under divergence form. 
                  This  will be   done  in \cite{BD2}. 
                                      
                         We will derive the local Lipschitz regularity for $W^{1,p}_{loc} $ solutions and $f\in L^\infty_{loc}$,  from the one for viscosity  solutions and $f$ continuous. 
                         
                          We  hope  that the method  here employed  could be used to treat the case $p\leq 4$ and $N\geq 3$ , not covered
                           at this day, to my knowledge,  by the results of Bousquet,  Brasco and Julin, \cite{BBJ}  for the  widely  degenerate equation   (\ref{eq3}), as well as to weaken the  regularity of $f$ in \cite{BBJ}, but of course the high degeneracy  of (\ref{eq3})  brings additional technical difficulties.  
                           
                  A  further question we ask  is : 
                    {\rm does the   ${\cal C}^1$  or ${\cal C}^{1, \beta} $ regularity holds, as in the case of the classical $p$Laplacian, \cite{T}, \cite{DB}}? A first step would consist in proving the ${\cal C}^1$ regularity when the right hand side is zero and then deduce from it the case $f\neq 0$ by methods as in \cite{IS}, \cite{BD3}, but even  in the case $f\equiv 0$  i have no intuition about the truthfullness of this result.   The usual methods  in the theory of viscosity solutions, (\cite{CC}, \cite{IL}),  cannot  directly be  applied   to the present case, one of the key argument  of  their proofs being the  uniform ellipticity of the operator.     Likewise, the methods of Figalli and Colombo  \cite{CF} to prove the regularity  outside of the degeneracy set of the operator,  suppose that this  set  is bounded, which is not the case here.  
                     One must find  new arguments.

                       Another probable extension  of the results included here  consists  in  proving that the  Lipschitz regularity still   holds  when  $f\in L^k$  for  $k>N$, using $L^k$-viscosity solutions (\cite{CKLS}) in place of ${\cal C}$-viscosity solutions.                          This  could be the object of a future work.      

  \bigskip
  
 As we saied before in particular we deduce from Theorem \ref{th1bis}  the following  result  : 
  
   \begin{theo}\label{th1}
      
     For any $p> 2$  and for all  $r<1$,  there exists $C$ depending on $( p, N, r)$ such that  for any $u$  a 
      ${\cal C}$-viscosity  (respectively weak), bounded  solution in $B_1$  of (\ref{eq1}) ,   with $f\in {\cal  C} (\overline{B_1})$ (respectively $f\in L^\infty (B_1)$),   
      $$ {\rm Lip}_{B_r}  u \leq C (|u|_{L^\infty(B_1)} + |f|_{L^\infty (B_1)} ^{1\over p-1}). $$

      \end{theo}

The plan of this paper is as follows  : 
 
 In Section 2, we recall  some basic facts about viscosity solutions. We  give the material for deducing from the Lipschitz regularity result for viscosity solutions and  a right hand side continuous,  that the same  holds true for weak solutions and $f\in L^\infty_{loc}$.    In Section 3 we prove  Lipschitz regularity  estimates between  viscosity sub- and super-solutions, ie the content of Theorem \ref{th1bis}.    
 \bigskip

 \section{Weak solutions and viscosity solutions}
 \subsection{About viscosity solutions}
 {\em Notations : In all the paper $B_r$ denotes the open ball of center $0$ and radius $r$. 
    $x$   (respectively $\left(\begin{array}{cc} x\\y\end{array}\right)$) denotes  a vector column in $\R^N$ (respectively a column in  $\R^{N}\times \R^N$), while 
    $^t x$, (respectively $(^t x, ^t y)$) denotes a  vector line in $\R^N$ (respectively a vector line  in $\R^{N}\times \R^N$).  For $x$ and $y\in \R^N$, we denote  the scalar product  of $x$ and $y$ by  $\langle x, y
\rangle $ or $^t x y$ or $x\cdot y$.  For $x\in \R^N$, $|x|$ denotes 
    the euclidian norm $|x| = (x \cdot x)^{1\over 2}  = (\sum_{i=1}^N |x_i|^2)^{1\over 2}$.    $S$ is the space of symmetric matrices on $\R^N$.  For $A\in S$, we define   the norm 
     $|A| = \sup _{x\in \R^N, |x|=1} |^t x A x|$ or equivalently $|A|= \sup_{1\leq i\leq N} |\lambda_i(A)| $  where  the $\lambda_i(A)$ are the eigenvalues of $A$.  We recall that $X\leq Y$ when $X$ and $Y$ are in $S$,  means that 
     $Y-X \geq 0$ i.e  for all $x\in \R^N$, $^t x (Y-X) x \geq 0$, or equivalently  $\inf_{1\leq i\leq N} \lambda_i (Y-X) \geq 0$.}

  Let us recall the definition of ${\cal C}$- viscosity solutions for Elliptic Second Order Differential Operators : 
  
   Let $F$ be continuous  on $\Omega \times \R \times \R^N\times S$,    where $\Omega$ denotes an open subset in $\R^N$.   We consider the "equation" 
   
   $$F(x, u, Du, D^2 u) = 0. $$ 
  
   \begin{defi}
   
 $u$,  lower-semicontinuous (LSC) in $\Omega$ is a  ${\cal C}$-viscosity super-solution of 
   $F(x, u, Du, D^2 u) = 0$ in $\Omega$ 
    if,  for any $x_o \in \Omega$ and any $\varphi$,   ${\cal C}^2 $ around $x_o$ which satisfies  
    $(u-\varphi)(x) \geq (u-\varphi)(x_o) =0 $  in a neighborhood of $x_o$, one has 
 $F(x_o, \varphi(x_o), D\varphi(x_o), D^2 \varphi(x_o)) \leq 0$. 
      
      $u$  upper-semicontinuous (USC), is a   ${\cal C}$-viscosity sub-solution  of  
      
      $F(x, u, Du, D^2 u) = 0$ in $\Omega$  if    for any $x_o \in \Omega$ and any $\varphi$,  
       ${\cal C}^2 $ around $x_o$ which satisfies  
    $(u-\varphi)(x) \leq (u-\varphi)(x_o) =0 $  in a neighborhood of $x_o$, one has 
 $F(x_o, \varphi(x_o), D\varphi(x_o), D^2 \varphi(x_o)) \geq 0$.

     $u$ is a solution if  it is both a super- and a sub-solution. 
            \end{defi}

       It is classical  in the theory of Second Order Elliptic Equations that one can   work with  semi-jets, and closed semi-jets  in place of ${\cal C}^2$ functions. For the convenience of the reader we recall their definition :   
       
        \begin{defi}        
          Let $u$ be an  upper semi-continuous function  in a neighbourhood of $\bar x$. Then we define the  super-jet $(q, X)\in \R^N \times S$  and we note 
          $ (q, X)\in J^{2,+}u(\bar x)$ if  there exists $r>0$ such that 
          for all $x\in B_r(\bar x)$, 
          $$u(x) \leq u(\bar x)+ \langle q, x-\bar x\rangle + {1\over 2} \ ^t (x-\bar x) X (x-\bar x) + o(|x-\bar x|^2).$$
          
           Let $u$ be a lower  semi-continuous  function  in a neighbourhood of $\bar x$. Then we define the sub-jet $(q, X)\in \R^N \times S$  and we note 
          $ (q, X)\in J^{2,-}u(\bar x)$ if  there exists $r>0$ such that
          for all $x\in B_r(\bar x)$, 
          $$u(x) \geq u(\bar x)+ \langle q, x-\bar x\rangle + {1\over 2} \ ^t (x-\bar x) X (x-\bar x) + o(|x-\bar x|^2).$$
        
        We also define  the "closed semi-jets" :
           
           \begin{eqnarray*}
            \bar J^{2,\pm } u(\bar x) &=&\{ (q, X), \exists \ (x_n,q_n, X_n), \ ( q_n, X_n) \in J^{2,\pm } u (x_n)\ \\
            &&{\rm and} 
           \ (x_n, q_n, X_n) \rightarrow (\bar x, q, X)\}.
           \end{eqnarray*}

          \end{defi}

       We  refer to the survey of Ishii \cite{I1}, and  to \cite{usr} for more  complete results about semi-jets:  The link between semi-jets and  test functions for sub- and super-solutions is the following : 
       
        {\em $ u$,  USC is  a sub-solution  if and only if for any $\bar x$ and for any  $(q, X) \in \bar J^{2,+} u(\bar x)$, then }
        \begin{equation}\label{altdef}
         F(\bar x , u(\bar x), q, X) \geq0
         \end{equation}
         and    the same with analogous changes is valid for super-solutions.

       Let us now recall 
          Lemma 9 in \cite{I1} and one of its consequences for the proofs in the present paper
         
             \begin{lemme}\label{lem1}
         Suppose that  $A$ is  a symmetric matrix on $\R^{2N}$ and that $U \in USC (\R^N)$, $V\in USC (\R^N)$ satisfy 
         $U(0)= V(0)$ and for all $(x,y)\in( \R^N)^2$
         $$U(x)+ V(y) \leq {1\over 2} (^tx,^ty)A \left(\begin{array}{c}
         x\\
         y\end{array}\right).$$
          Then for all $\iota>0$ there exist $X^U_\iota \in S$, $X^V_\iota \in S$ such that 
          $$(0, X^U_ \iota) \in \bar J^{2,+} U(0), \ (0, X^V_\iota)\in \bar J^{2,+} V(0)$$
           and 
           \begin{equation}\label{eqiota} -({1\over \iota} + |A|) \left(\begin{array}{cc}
           I &0\\
           0&I\end{array}\right) \leq \left(\begin{array}{cc}
           X^U_\iota&0\\
           0& X^V_\iota
           \end{array}\right)\leq (A+\iota A^2).
           \end{equation}
           \end{lemme}
                     The proof of Lemma \ref{lem1} uses the approximation of $U$ and $V$ by Sup and Inf convolution. 
         This lemma has the following consequence  for the results of this paper : 
         
         \begin{lemme} \label{lem2}
         
                      Suppose that $u$ and $v$ are respectively  USC and LSC and satisfy for some constant $M>1$ and   for some function $\Phi$ which is ${\cal C}^2$ around $(\bar x, \bar y)$
            \begin{eqnarray*}
            u(x)-v(y) -M  |x-x_o|^2 -M  |y-x_o|^2- \Phi(x,y)&\leq & u(\bar x)-v(\bar y) -M |\bar x-x_o|^2\\
            & -&M |\bar y-x_o|^2
            - \Phi(\bar x, \bar y).
            \end{eqnarray*}
           Then for any $\iota$,   there exist $X_\iota ,Y_\iota$ such that 
               $$ (D_1 \Phi(\bar x, \bar y)+ 2M (\bar x-x_o), X_\iota) \in \bar J^{2,+} u(\bar x), \ 
               (-D_2 \Phi(\bar x, \bar y) -2M (\bar y-x_o), -Y_\iota) \in \bar J^{2,-} v(\bar y)$$
                with 
                    \begin{equation}\label{eqiota} -({1\over \iota} + |A|+1 ) \left(\begin{array}{cc}
           I &0\\
           0&I\end{array}\right) \leq \left(\begin{array}{cc}
           X_\iota-2MId&0\\
           0& Y_\iota-2M Id
           \end{array}\right)\leq (A+\iota A^2)    +  \left(\begin{array}{cc}
           I &0\\
           0&I\end{array}\right)         \end{equation}
            and $A = D^2 \Phi(\bar x, \bar y)$.
            \end{lemme}

               \begin{proof}
              By Taylor's formula    at the order $2$ for $\Phi$,  for all $\epsilon  >0$  there exists   $r>0$ such that  
               for $|x-\bar x| ^2 + |\bar y -y|^2 \leq r^2$
                \begin{eqnarray*}
                u(x)-\langle D_1\Phi(\bar x, \bar y)&+&2M(\bar x-x_o), x-\bar x\rangle -v(y)- \langle D_2 \phi(\bar x, \bar y)+2M(\bar y-x_o)  , y-\bar y\rangle  \\
                &-&u(\bar x)+v(\bar y) \\
                &\leq &{1\over 2} \left(^t (x-\bar x), ^t (y-\bar y) \right)(D^2 \Phi(\bar x, \bar y)+\epsilon  Id )  \left( \begin{array}{c} x-\bar x\\
                y-\bar y\end{array}\right)\\
                &+& M   (|x-\bar x|^2 + |y-\bar  y|^2). 
                \end{eqnarray*}

          We define 
                      $$U(x) = u(x+ \bar x)-\langle D_1\Phi(\bar x, \bar y)+2M(\bar x-x_o), x\rangle -u(\bar x) -M |x|^2,$$
                 $$ V(y) = -v(y+ \bar y) - \langle D_2 \Phi(\bar x, \bar y)+2M(\bar y-x_o)  , y\rangle + v(\bar y)-M |y|^2 $$
                in the  closed ball  $|x-\bar x| ^2 + |\bar y -y|^2 \leq r^2$,  extend $U$ and $V$ by some convenient negative constants   in the complementary (see  \cite{I1} for details ) and apply Lemma \ref{lem1}. Note that

                $$ (0, X_\iota^U) \in \overline{J}^{2,+} U(0), \ (0, X_\iota^V ) \in \overline{J}^{2,-} V(0)$$ is equivalent to 
                
                 $$(D_1 \Phi(\bar x, \bar y) + 2M (\bar x-x_o), X_\iota^U + 2MId) \in  \overline{J}^{2,+} u(\bar x)$$
                  and 
                  $$ (-D_2 \Phi (\bar  x, \bar y)-2M (\bar y-x_o), -X_\iota^V - 2MId )\in \overline{J}^{2,-} v(\bar y)$$
                  hence one obtains    that 
                  for any $\iota$ there exists 
                  $(X_\iota, Y_\iota) $ such that

                  $(D_1 \Phi(\bar x, \bar y)+2M(\bar x-x_o), X_\iota )\in \bar J^{2,+} u(\bar x)$, 
                  $ (-D_2 \Phi(\bar x, \bar y) -2M (\bar y-x_o), -Y_\iota ) \in \bar J^{2,-}v(\bar y)$
                   and  taking $\epsilon $ such that 
                   $2 \epsilon  \iota |A| +  \epsilon+ \iota (\epsilon )^2 < 1$    
                     \begin{equation}\label{eqiota} -({1\over \iota} + |A|+1 ) \left(\begin{array}{cc}
           I &0\\
           0&I\end{array}\right) \leq \left(\begin{array}{cc}
           X_\iota-2MId&0\\
           0& Y_\iota-2M Id
           \end{array}\right)\leq (A+\iota A^2)+  \left(\begin{array}{cc}
           I &0\\
           0&I\end{array}\right)  .
           \end{equation}
           
            where 
            $A = D^2 \Phi(\bar x, \bar y)$. 
            \end{proof}

            In the sequel we will use  Lemma \ref{lem2} with  $\Phi(x,y)$ of the form 
             $ \Phi(x,y) =Mg (x-y) =M \omega (|x-y|)$, where $g$ is defined on $\R^N$  and $\omega$  on $\R^+$ will be   defined later and then  noting $H_1(x) = D^2 g(x)$,  defining $\iota = {1\over 4M |H_1(x)}$, and $\tilde H(x) = H_1 (x)+ 2\iota  H_1^2(x)$ :  
             $$ A =M \left(\begin{array}{cc} H_1(\bar x-\bar y)& -H_1(\bar x-\bar y)\\
             -H_1(\bar x-\bar y)& H_1(\bar x-\bar y)\end{array} \right) \ {\rm and } \ A + \iota A^2 =    M\left(\begin{array}{cc} \tilde H(\bar x-\bar y)& -  \tilde H(\bar x-\bar y)\\
             -\tilde H(\bar x-\bar y)&  \tilde H(\bar x-\bar y)\end{array} \right)$$ 
              Note that $|A| = 2M |H_1(\bar x-\bar y)|$.                       
                In all the situations later, (\ref{eqiota}) has the following consequence  for $X:=X_\iota$ and $Y:= Y_\iota$, 
              
              \begin{equation}\label{majnorm}
              |X-(2M +1)\ Id| + |Y-(2M+1) \ Id | \leq  6 M|D^2g(\bar x-\bar y)|
              \end{equation}

              \bigskip
                             
                   In the rest of the paper we will consider    the operators 
                   $$F(x, u, q, X) = \sum_i |q_i|^{p-2}  X_{ii} -f(x) := F(q, X)-f(x) $$
                  where $q_i =\vec  q\cdot e_i$ and $X_{ij} = ^t e_i Xe_j$, $e_i$ is some  given  orthonormal basis in $\R^N$, $f$ is continuous.    
                   In the sequel we suppose  known that the weak solutions are continuous, see for example   \cite{BC}.  This  permits  to use the  above definition of viscosity solutions, and not its generalization  to bounded functions $u$  which makes use of lower semi-continuous or upper semi -continuous envelope of $u$, see \cite{I1}. 
                   
                   \subsection{Weak solutions are Viscosity solutions} 
            In this section we want to show how   one can deduce the Lipschitz regularity result for weak solutions,  from the regularity result for viscosity solutions, i.e the half part of Theorem  \ref{th1}. 
            
        We begin to recall  the   following  comparison theorem  for weak solutions, 
   \begin{theo}\label{comp}
              
               Suppose that $\Omega $ is a    bounded  ${\cal C}^1$ open subset in $\R^N$, that  $u$  and  $v$ are in $W^{1,p} (\Omega)$ and  satisfy in the distribution sense  
               $ \tilde \Delta_p u \geq  \tilde \Delta_p v$.  
Suppose that $u\leq v$ on $\partial \Omega$, then 
$u\leq v$ in $\Omega$.

\end{theo}
\begin{proof}   
  Since $(u-v)^+ \in W_0^{1,p} (\Omega)$,  there exists $\varphi_\epsilon\in {\cal D} (\Omega)$, which converges to $(u-v)^+$ in $W^{1,p} (\Omega)$, $\varphi_\epsilon $  can be chosen $\geq 0$.  Multiply the  difference $\tilde \Delta_p u -\tilde  \Delta_p v$ by $\varphi_\epsilon $ and use the definition of the derivative in the distribution sense. One obtains
 $$-\sum_{i=1}^N \int_\Omega  (|\partial_i u|^{p-2} \partial_i u-|\partial_i v|^{p-2} \partial_i v)\partial_i \varphi_\epsilon = \int_\Omega (\tilde \Delta_p u-\tilde \Delta_p v) \varphi_\epsilon \geq 0. $$
  The left hand side tends to 
   $$-\sum_{i=1}^N \int_\Omega  (|\partial_i u|^{p-2} \partial_i u-|\partial_i v|^{p-2} \partial_i v)\partial_i (u-v)^+\leq 0$$ therefore 
   $\int_{u-v \geq 0}\sum_{1\leq i\leq N} (|\partial_i u|^{p-2} \partial_i u-|\partial_i v|^{p-2} \partial_i v)\partial_i (u-v)=0$, 
   hence $\partial_i ((u-v)^+) = 0$ for all $i$, finally $u\leq v$ in $\Omega$. 
   \end{proof}
 
\bigskip

   Let us take a few lines to  motivate the following propositions. 
    We want to have a Lipschitz estimate depending on the $L^\infty$ norm of $u$ and  $f$, when $u$ is a weak solution. 
      A natural idea is to regularize $f$, hence use $f_\epsilon \in {\cal C} (\overline{B_1})$ which tends to $f$ in $L^\infty$ weakly,   define $u_\epsilon$ which is a solution of the Dirichlet problem  associated to (\ref{eq1}) with a  right hand side $f_\epsilon$ and  the   boundary data $u$ on $\partial B_1$.  
      In order to apply the results in Section 3, we  will use the fact that $u_\epsilon$ is also a viscosity solution,  (see \cite{BK} or   Proposition \ref{viscoweak} below). 
      Since the  uniform  Lipschitz estimate  in Section 3 for  $u_\epsilon$,  depends on the $L^\infty$ norms of $u_\epsilon$ and $f_\epsilon$, we need, in order to pass to the limit and obtain some estimate depending on  the $L^\infty$ norms of $u$ and $f$, to "control" the $L^\infty$ norm of $u_\epsilon$ by those of $u$,  this is what we do in Proposition \ref{prop1} now.     
   \begin{prop}\label{prop1}

 There exists some constant $C$ depending only on $(p, N)$,  such that for any  $f\in L^\infty(B_1)$ and $v\in L^\infty(\partial B_1)\cap W^{1-{1\over p},p}  (\partial B_1)$,   if  $u$ is a   weak solution of the Dirichlet problem
$$\left\{ \begin{array}{lc}
\tilde \Delta_p u =(p-1)  f & {\rm in} \  B_1\\
 u = v & {\rm on} \ \partial  B_1
 \end{array}\right.$$
 then 
 \begin{equation}\label{eqinfinie}|u|_{L^\infty (B_1)} \leq C (|f|_{L^\infty(B_1)}^{1\over p-1}+ |v|_{L^\infty(\partial  B_1)}).
 \end{equation}
 \end{prop}

\begin{proof} 

Let  $d$ denote the distance to the boundary,  say $d(x) = 1-|x|$,  and  let $h$ be defined by 
$h(x) = |v|_{L^\infty(\partial B_1)} + M  \left(1-{1\over (1+ d(x))}\right) $.

We want to prove that   as soon as $M$ is  large as $|f|_\infty^{1\over p-1}$, and depending on $(N,p)$,  $h$   is a weak super-solution of (\ref{eq1})  in $B_1$, more precisely  $h\in W^{1,p}(B_1
)$ and  $\tilde \Delta_p h\leq -(p-1)|f| _\infty$,  in the distribution  sense.

Note that   $h$ is ${\cal C}^2$ except at zero,  $|D h|\leq M$,    and for $x\neq 0$,  $\tilde \Delta_p h(x) \leq -(p-1) |f|_\infty$.  Suppose for a while that this last assertion  has been proven, and let us derive from it that $h$ is a weak super-solution.  Let  $\varphi \in {\cal D} (B_1)$, 
 $\varphi \geq 0$ :
 \begin{eqnarray*}
\langle \tilde \Delta_p  h, \varphi\rangle &=& -\int_{B_1} \sum_{i=1}^N |\partial_i h|^{p-2} \partial_i h \partial_i \varphi =
 \lim_{\epsilon \rightarrow  0} -\int_{B_1\setminus \overline{B(0, \epsilon)}}  \sum_1^N|\partial_i h|^{p-2} \partial_i h \partial_i \varphi\\
&=& \lim_{\epsilon\rightarrow  0} \int_{B_1\setminus  \overline{B(0, \epsilon)}}  (p-1)  \sum_{i=1}^N |\partial_i h|^{p-2}\partial_{ii} h \varphi - \int_{\partial B(0, \epsilon)}  \sum_{i=1}^N  |\partial_i h|^{p-2} \partial_i h n_i \varphi \\
&\leq &-(p-1)\lim_{\epsilon \rightarrow 0} |f|_\infty \int_{B_1\setminus  \overline{B(0, \epsilon)}} \varphi +M^{p-1}  \lim_{\epsilon\rightarrow 0} \int_{\partial B(0, \epsilon)} \varphi \\
&\leq &-(p-1)|f|_\infty \int_{B_1} \varphi.
\end{eqnarray*}
 There remains to prove that  for $x\neq 0$,  and  for $M$ chosen conveniently, $\tilde \Delta_p h(x) \leq -(p-1) |f|_\infty$. One has for $x\neq 0$
 $$D h=  M{ D d\over (1+ d)^2},\ D^2 h = {-2 M(D d\otimes D d) + (1+ d) D^2d\over (1+ d)^3}$$
 and then  since $D_{ii} d \leq 0$,  by  the concavity of $d$,  
 
 \begin{eqnarray*}
(p-1) \sum_{i=1}^N |\partial_i h |^{p-2} \partial_{ii} h &\leq & -2M^{p-1} (p-1){\sum_{i=1}^N |\partial_i d|^p\over (1+d)^{2p+1} } \\
 &\leq & -2^{-2p} M^{p-1} (p-1)N^{{1-{p\over 2}}}.\\
 \end{eqnarray*}
  Then we choose $M$ so that 
$M^{p-1} 2^{-2p} N^{1-{p\over 2}}> |f|_\infty$ and get $\tilde \Delta_p h(x) \leq -(p-1)|f|_\infty$.  Using Theorem \ref{comp} one gets that 
  $u\leq h$ in $B_1$, which is the desired conclusion.  Replacing $h$ by $-h$,  one sees that $\tilde \Delta _p (-h) \geq (p-1) |f|_\infty$ and then $-h$ is a sub-solution of (\ref{eq1}), hence $u\geq -h$ in $\Omega$. 
  
 \end{proof}

We now   recall the following  local result, \cite{BK} : 
 \begin{prop}\label{viscoweak} Suppose that $u$ is a weak ($W^{1,p}$) solution of $\tilde \Delta_pu= (p-1)  f$  in some open set ${\cal O}$ with 
$ f\in {\cal C} ({\cal O})$,  then $u$ is a ${\cal C}$-viscosity solution of the same equation in ${\cal O}$.
\end{prop} 

\begin{proof} 
The proof is made in   \cite{BK}, but it is written here for the reader's convenience.  

 We do the super-solution case.    Take any $x_o\in {\cal O}$ and some ${\cal C}^2$ function $\varphi$ such that 
$(u-\varphi)(x) \geq (u-\varphi)(x_o) = 0$ in a neighbourhood of $x_o$.
We can assume  the inequality to be  strict  for $x\neq x_o$, by replacing $\varphi$ by $x\mapsto \varphi(x) -|x-x_o|^4$.      
  Assume by contradiction that   for some $\epsilon >0$ 
$$(p-1)\sum_{i=1}^N|\partial_i \varphi (x_o)|^{p-2}{\partial_{ii} \varphi(x_o)} \geq (p-1) f(x_o)+\epsilon,$$
then by continuity this is also true (up to changing $\epsilon$ )  in a neighbourhood $B_r(x_o)$, with $\overline{B_r(x_o)} \subset {\cal O}$. 
Let $m = \inf _{\partial B_r(x_o)} (u-\varphi)$ and $\phi = \varphi + {m\over 2}$ which satisfies 
$u-\phi >0$ on $\partial B_r(x_o)$ and $\phi (x_o) > u(x_o)$. We also have
 
 $(p-1)\sum_{i=1}^N |\partial_i \phi (x)|^{p-2}{\partial_{ii} \phi(x)} \geq(p-1)  f(x)+\epsilon $ in $B_r(x_o)$.
Multiplying by $(\phi-u)^+$,  and integrating over $B_r$, we obtain 
$$-\int_{B_r(x_o)\cap \{\phi-u>0\}} \sum_{i=1}^N|\partial_i \phi|^{p-2}\partial_i \phi\partial_i (\phi-u)\geq  (p-1)\int_{B_r(x_o)}  (f+\epsilon) (\phi-u)^+.$$
On the other hand, since $u$ is a weak solution, 
$$\int_{B_r(x_o)\cap \{\phi-u>0\}} \sum_{i=1}^N |\partial_i u|^{p-2} \partial_i u \partial_i (\phi-u) = (p-1) \ \int_{B_r(x_o)}   f(\phi-u)^+.$$
Adding  the two equations, one obtains 
 $$-\int_{B_r(x_o)} \sum_{i=1}^N\left( |\partial_i \phi|^{p-2}\partial_i \phi-|\partial_i u|^{p-2}{\partial_i u} \right) \partial_i (\phi-u)^+\geq  (p-1) \epsilon \int_{B_r(x_o)}  (\phi-u)^+. $$
This is a contradiction, because the right hand side is positive, while the left hand side is non positive.
This proves that $u$ is a viscosity super-solution. We could do the same with obvious changes to prove that  a weak solution 
 is  a viscosity sub-solution. 
\end{proof}

We derive from this the  result in Theorem \ref{th1}  for weak solutions, once we know it for viscosity solutions : 
               
Suppose that  $f\in L^\infty (B_1)$ and  that $u$ is a weak solution of $
\tilde \Delta_p u =(p-1) f $   in $ B_1$, $u\in L^\infty (B_1)$. 
 Let  $f_\epsilon \in {\cal C} (\overline{B_1})$, $f_\epsilon \rightharpoonup f$ in $L^\infty (B_1)$ weakly,  $|f_\epsilon|_\infty \rightarrow |f|_\infty$, and  let $u_\epsilon$ be the unique weak   solution of 
 $$\left\{ \begin{array}{lc}
 \tilde \Delta_p u_\epsilon  =(p-1) f_\epsilon & {\rm in} \ B_{1}\\
 u_\epsilon = u & {\rm on} \ \partial B_{1}. 
 \end{array}\right.$$
  which makes sense since $u_{\mid {\partial B_1}} \in W^{1-{1\over p},p} (\partial B_1)$.  It is equivalent to say that $u_\epsilon$    satisfies 
  $ J_\epsilon (u_\epsilon) = \inf_{v-u\in W_o^{1,p} (B_1)} J_\epsilon (v)$
 where  $J_\epsilon (v) = {1\over p} \int_{B_1}\sum_{1}^N  |\partial_i v|^p+(p-1)\int_{B_1} f_\epsilon v$,  $J (v) = {1\over p} \int_{B_1} \sum_1^N |\partial_i v|^p+(p-1)\int_{B_1} f v$. 
    It is clear that  $(u_\epsilon)_\epsilon$ is bounded in $W^{1,p}(B_1)$, and that 
   ${\rm limsup}_{\epsilon \rightarrow 0} \inf_{\{v-u\in W_o^{1,p}\}} J_\epsilon(v)
    \leq \inf_{\{v-u\in W_o^{1,p}\}} J(v)$.      Hence  one can extract from it  a subsequence 
  which converges   weakly  in $W^{1,p} (B_1)$  to  some $\bar u$.  Note that by Poincar\'e's inequality,  since $u_\epsilon-\bar u \in W_o^{1,p}$,  $u_\epsilon -\bar u$ tends to zero in $L^p$ strongly. Hence $\bar u$ satisfies   
  $ J(\bar u ) \leq \liminf J_\epsilon (u_\epsilon)$. 
   Finally $\bar u$  is a 
   minimizer for  $J$, with $\bar u = u$ on $\partial B_1$,  hence    by uniqueness  $\bar u =u$  and the convergence of $u_\epsilon$ to $u$ is strong in $W^{1,p}$. By   the results in Section 3,  for a  viscosity solution and  a right hand side continuous and bounded,    for all $r$ 
  there exists some   constant $C(N,p,r)$  such that  for any $\epsilon$,  since $u_\epsilon$ is a ${\cal C}$-viscosity solution in $B_1$,
 $${\rm Lip}_{B_r}  (u_\epsilon) \leq C(N,p,r)(|u_\epsilon|_{L^\infty(B_1)}+ |f_\epsilon|_{L^\infty(B_1 )}^{1\over p-1}).$$
 Using (\ref{eqinfinie}) for $u_\epsilon$ and passing to the limit,  one gets the estimate
    ${\rm Lip}_{B_r}(u) \leq C(N,p,r)(|u|_{L^\infty(B_1)} + |f|_{L^\infty(B_1)} ^{1\over p-1})$.

\begin{rema}
 We  also have the alternative result :

 There exists some constant $C(N, p, r)$ such that for any $u$ a weak solution of (\ref{eq1}) in $B_1$, 
 ${\rm Lip}_{B_r} (u) \leq  C(N,p,r)(|u|_{W^{1,p}(B_1)} + |f|_{L^\infty(B_1)} ^{1\over p-1})$.
 
 This result can  directly  be deduced from the previous one  by using the $L^\infty_{loc}$ estimate 
 $$ |u|_{L^\infty(B_{1+r\over 2})} \leq C (|u|_{W^{1,p} (B_1)} + |f|_\infty^{1\over p-1})$$
  which can be derived from the results in \cite{G}, \cite{BC}.
  \end{rema}

\section{ Proof of Theorem \ref{th1bis}}

 From now we assume that $f$ and $g$ are continuous in $B_1$ and that 
 $u$ satisfies in the viscosity sense 
     $$ \sum_i |\partial_i u|^{p-2} \partial_{ii} u \geq  f $$
                         and $v$ satisfies in the viscosity sense 
                         $$ \sum_i |\partial_i v|^{p-2} \partial_{ii} v  \leq  g .  $$

The notation $c$,  and   $c_i$ will always denote some positive constants which depend only on $r, p, N$,  the H\"older's exponent $\gamma$  when it intervenes and of universal constants.  Proving in a first time  some  H\"older's estimate   for any exponent $\gamma$ is necessary  to get the Lipschitz  estimate. 
\subsection{Material for the proofs}

In all the section, $\omega$ denotes some   continuous  function on $\R^+$,  such that $\omega(0)=0$, $\omega$ is ${\cal C}^2$ on $\R^{+\star}$    and  $\omega (s) >0$,   $\omega^\prime (s)>0$ and $\omega^{\prime \prime} (s) <0$ on $]0,1[$. Let  $g$ be  the radial function 
 $$g(x) = \omega (|x|).$$
 
  Then  for $|x|<1$,
 $Dg (x)= \omega^\prime (|x|) {x\over |x|},  $
  and 
  $$D^2 g (x) = \left(\omega^{\prime \prime } ( |x|) -{\omega^\prime ( |x|)\over  |x|}\right) {x\otimes x\over  |x|^2} +  {\omega^\prime ( |x|)\over  |x|} Id. $$
  We denote  by $ H_1(x)$ the   symmetric matrix with  entries  $\partial_{ij } g(x)$, and for $\iota \leq    {1\over 4 | H_1(x)| } $, we define 
  $ \tilde H =  H_1+2 \iota  H_1^2$. With that choice of $\iota$ there exist constants $\alpha_H\in ]{1\over 2} , {3\over 2}], \  \beta_H \geq {1\over 2}$ such that 
   \begin{equation} \label{alphabetaH}
   \tilde H (x)=  \left(\beta_H\omega^{\prime \prime } ( |x|) -\alpha_H{\omega^\prime ( |x|)\over  |x|}\right) {x\otimes x\over  |x|^2} + \alpha_H {\omega^\prime ( |x|)\over  |x|} Id. 
   \end{equation} 
When $p>4$, and for any  $|x|<1$, $x\neq 0$,   we will use a number   $\epsilon >0$,  
 ( that we will  make precise  depending   on the  H\"older and Lipschitz cases  )  for which  we define    $I(x, \epsilon) = \{ i\in [1,N] , |x_i | \geq  |x|^{1+\epsilon} \} $. When no ambiguity arises, we will denote  it  
  $I$ for simplicity. We define    the vector 
   \begin{equation}\label{w}  w  =\left\{ \begin{array}{lc}
    \sum_1^N|x_i|^{2-p\over 2} x_i e_i &\ {\rm  if } \ p\leq 4\\
   \sum_{i\in I} |x_i|^{2-p\over 2} x_i e_i &\ {\rm if} \  p>4, \ {\rm when} \  I(x, \epsilon) \neq \emptyset .
   \end{array}\right.
   \end{equation}
    Note that if $p\leq 4 $,  
\begin{equation}\label{w1}
| w|^2 \leq  |x|^{4-p} N^{p-2\over 2}, 
\end{equation} 
 while  if $p>4$
\begin{equation}\label{w2} | w|^2 \leq  \#I(x, \epsilon) |x|^{(4-p)(1+\epsilon)}.
\end{equation}  
 Furthermore,   $$|x|^2-\sum_{i\in I} |x_i|^2= \sum _1^N |x_i|^2 -\sum_{i\in I} |x_i|^{2} \leq N |x|^{2+2\epsilon} .$$
 We  also define the diagonal matrix $\Theta(x)$ with  entries 
  $\Theta_{ij}(x) =\left\vert { \omega^\prime (|x|) x_i\over |x|}\right\vert^{p-2\over 2}\delta_i^j$, where $\delta_i^j$ denotes the Kronecker symbol, 
   and the matrix $$H(x)  = \Theta (x) \tilde H (x)\Theta (x).$$
   
   \begin{prop}\label{prop4}

  \noindent  1) If $p\leq 4$,  for all $x\neq 0$,   $|x|<1$,   $H(x)$
 has at least one eigenvalue less than 
 \begin{equation}\label{pleq4} N^{1-{p\over 2}}\  \beta_H\omega^{\prime\prime} (|x|)(\omega^\prime (|x|) )^{p-2}.
 \end{equation}
 2)  If $p>4$, for all $x\neq 0$, $|x|<1$,   for any   $\epsilon >0$    such that $I(x, \epsilon) \neq \emptyset$,  and  such that 
 \begin{equation}\label{eqNepsilon}
\beta_H\omega^{\prime\prime} (|x|) (1-N |x|^{2\epsilon} )+\alpha_H N  |x|^{2\epsilon}{ \omega^{\prime }(|x|)\over |x|} \leq  {\omega^{\prime\prime} (|x|)\over 4}  <0 , 
  \end{equation} 
 then 
  $H(x)$ possesses at least one eigenvalue less than 
   \begin{eqnarray}\label{p>4}
 {1-N |x|^{2\epsilon}\over \# I(x, \epsilon)} (\omega^\prime (|x|))^{p-2}|x|^{(p-4)\epsilon}  {\omega^{\prime\prime} (|x|)\over 4}
   \end{eqnarray}
   
    \end{prop}

     \begin{proof} 
     
      Using the definitions  of $ H_1$  and $H$ one has 
      \begin{eqnarray*}
      H_{ij}(x) &=&  \left({\omega^\prime (|x|)\over |x|} \right)^{p-2 }\left(\beta_H\omega^{\prime \prime } ( |x|) -\alpha_H{\omega^\prime ( |x|)\over  |x|}\right) |x_i|^{p-2\over 2} |x_j|^{p-2\over 2}  {x_ix_j\over  |x|^2}\\
      &+& \alpha_H \left({\omega^\prime (|x|)\over |x|} \right)^{p-1} |x_i|^{p-2} \delta_i^j.
      \end{eqnarray*}

           Let $ w$ be defined above,  
           
           - For $p\leq 4$, using the definition of $ w$  in (\ref{w}) 
           $$^t  w H(x)   w =\beta_H\left({\omega^\prime(|x|)\over |x|}\right)^{p-2}   \omega^{\prime \prime} (|x|)|x|^2, 
          $$
           and then  using  estimate (\ref{w1})
           $${^t w  H (x) w \over | w|_2^2} \leq \beta_HN^{2-p\over 2} \left( \omega^\prime(| x|)\right)^{p-2} \omega^{\prime \prime} (|x|)\leq {1\over 2}   N^{2-p\over 2} \left( \omega^\prime(| x|)\right)^{p-2} \omega^{\prime \prime} (|x|), $$
     -while if $p>4$, 
            \begin{eqnarray*}
            ^t  w H (x) w &=&\left({\omega^\prime(|x|)\over |x|}\right)^{p-2}  \beta_H\omega^{\prime \prime} (|x|) {(\sum_{i\in I} |x_i|^2)^2\over  |x|^2} \\
            &+& \alpha_H\left({\omega^\prime(|x|)\over |x|}\right)^{p-1}(-{(\sum_{i\in I} |x_i|^2)^2
\over |x|^2}+\sum_{i\in I} |x_i|^2)  , 
\end{eqnarray*}

and then  if $p> 4$, and $I( x, \epsilon ) \neq \emptyset$, using  (\ref{eqNepsilon}), (\ref{alphabetaH})  and  (\ref{w2})   
  \begin{eqnarray*}
&& {^t  w H (x) w\over  |w|_2^2} \\
&\leq&  \left({\omega^\prime(|x|)\over |x|}\right)^{p-2} { \sum_{i\in I}  |x_i|^2 \over \#I(x, \epsilon) |x|^{(4-p)(1+\epsilon)} }( \beta_H\omega^{\prime \prime} (|x|)(1-N |x|^{2\epsilon})+  \alpha_HN  \omega^\prime (|x| ) |x|^{-1+2\epsilon})\\
 &\leq &{(1- N|x|^{2\epsilon}) \over \#I(x, \epsilon)}   \left(\omega^\prime(|x|)\right)^{p-2} |x|^{(p-4)\epsilon}  ({1\over 2}  \omega^{\prime \prime} (|x|)(1-N |x|^{2\epsilon})+  N {3\over 2}  \omega^\prime (|x| ) |x|^{-1+2\epsilon})\\
 &\leq &  {1-N |x|^{2\epsilon}\over \# I(x, \epsilon)} (\omega^\prime (|x|))^{p-2}|x|^{(p-4)\epsilon}  {\omega^{\prime\prime} (|x|)\over 4}  \end{eqnarray*}
      \end{proof}

        \bigskip
         We derive now from these last observations and from Proposition \ref{prop4} the following 
                 
                  \begin{prop}\label{prop5}
                  Suppose that $\omega$, $g$,  $\Theta(x)$,  and $ H_1(x)$ are as   in Proposition \ref{prop4} and that   for some $M>1$, for  $\iota = {1\over 4M |H_1(x)|}$ and $\tilde H(x) = H_1(x)+ 2 \iota H_1^2(x)$, $(X, Y)$ satisfy 
                   \begin{eqnarray} \label{eqXL} -6 M| H_1(x)| \left(\begin{array}{cc}
           I &0\\
           0&I\end{array}\right) &\leq& \left(\begin{array}{cc}
           X-(2M+1)Id&0\\
           0& Y-(2M+1)  Id
           \end{array}\right)\nonumber \\
           &\leq& M \left(\begin{array}{cc}
           \tilde H(x)&- \tilde H(x)\\
           - \tilde H(x)&  \tilde H(x)\end{array}\right)
           \end{eqnarray} 
            1) Then 
            $$\Theta(x)(X+Y-2(2M+1)  Id ) \Theta(x) \leq 0, $$
             consequently 
             $M^{p-2}\Theta(x) (X+Y) \Theta(x) \leq 2(2M+1)M^{p-2}   |\Theta(x) |^2 Id$,  and then 
             for any $i\in [1,N]$ 
             \begin{equation}\label{autresvp1}
               \lambda_i (M^{p-2} \Theta(x) (X+Y) \Theta(x) ) \leq  2(2M+1)M^{p-2}   |\Theta(x)|^2\leq 6M ^{p-1} |\Theta(x)|^2.
               \end{equation}  
               2) The smallest eigenvalue satisfies 
               \begin{equation} \label{mu1pleq4}
               {\rm If } \ p\leq 4,\ 
  \lambda_1 (M^{p-2} \Theta(x) (X+ Y-2(2M+1) Id ) \Theta(x) )\leq 2M^{p-1} {N^{2-p\over 2}}  \left( \omega^\prime(| x|)\right)^{p-2} \omega^{\prime \prime} (|x|).
  \end{equation}
   While if $p\geq 4$
      \begin{eqnarray} \label{mu1pgeq4}
 && \lambda_1 (M^{p-2}\Theta(x) (X+ Y-2(2M+1)  Id) \Theta(x) )\\
  &\leq & M^{p-1} {(1- N|x|^{2\epsilon}) \over \#I(x, \epsilon)}  \left(\omega^\prime(|x|)\right)^{p-2} |x|^{(p-4)\epsilon}  {\omega^{\prime \prime} (|x|) }  \nonumber \end{eqnarray}  
\end{prop}

\begin{proof}
 
 Note that (\ref{eqXL}) implies  that 
   $X-(2M+1) Id , Y-(2M+1)  Id  \geq  -6 M| H_1(x)| Id$
    and
   $X-(2M+1) Id, Y-(2M+1) Id \leq M |\tilde  H(x)| Id\leq {3 M\over 2} |H_1(x)|  Id $ hence 
    \begin{equation}\label{majnorm}
     |X-(2M+1) Id| , |Y-(2M+1) Id | \leq  6M  |H_1(x)| = 6M | D^2 g(x)|. 
     \end{equation}

  Muliplying equation (\ref{eqXL})  by $ \left(\begin{array}
        {cc}  \Theta(x) &0\\
        0& \Theta(x)
        \end{array}\right)$ on the right  and on the left  one gets 
 
    \begin{eqnarray*}
       \left(\begin{array}
        {cc}  \Theta(x) &0\\
        0& \Theta(x)
        \end{array}\right) &&\left(\begin{array}{cc} 
    X-(2M+1) \ Id &0\\
    0& Y-(2M+1)\ Id 
    \end{array}\right) \left(\begin{array}
        {cc} \Theta(x) &0\\
        0& \Theta(x)\end{array}\right)\\
        &\leq& M \left(\begin{array}
        {cc} \Theta(x) &0\\
        0&\Theta(x)\end{array}\right)\left(\begin{array}{cc} 
     \tilde H(x)&-\tilde  H(x)\\
    - \tilde H(x)&  \tilde H(x)
    \end{array}\right)  \left(\begin{array}
        {cc} \Theta(x) &0\\
        0&\Theta(x) \end{array}\right)\\
        &=& M \left(\begin{array}{cc} 
     H(x)&- H(x)\\
    -  H(x)&  H(x)
    \end{array}\right) 
               \end{eqnarray*}
               where $H(x) = \Theta(x) \tilde H(x) \Theta(x) = \Theta(x) (H_1(x) + {1\over 2 M |H_1(x)|} H_1^2(x))\Theta(x) $.

  To prove 1) let $ \vec v$ be any vector then multiplying by $(^t \vec v, ^t \vec v)$ on the left  of the previous inequalities  and by  $ \left(\begin{array}{cc}\vec v\\\vec v\end{array} \right)$ on the right  one gets that 
  $ ^t \vec v (\Theta(x) (X+Y-2(2M+1) Id ) \Theta(x)) \vec v \leq 0$
   which yields the desired result.     To prove 2)  using (\ref{pleq4}) let $e$ be an eigenvector for  $H(x)$, for  some eigenvalue less than
    $\beta_HN^{2-p\over 2} \left( \omega^\prime(| x|)\right)^{p-2} \omega^{\prime \prime} (|x|)$, let us multiply the right hand side of the previous inequality by 
   $(^t e, -^t e)$ on the left  and by $ \left(\begin{array}{cc}e\\-e\end{array} \right)$ on the right  one gets 
   $$ ^t e\   \Theta(x) (X+Y-2(2M+1) Id) \Theta(x)\  e \leq 4  ^t e  H(x) e $$  and then using (\ref{mu1pleq4})
    one gets 
    $$\lambda_1 (M^{p-2}  \Theta(x) (X+Y-2(2M+1) Id)\Theta(x) )\leq 2 M^{p-1}{N^{2-p\over 2}}  \omega^{\prime \prime }(|x|)  \omega^\prime (|x|)^{p-2} $$
    which yields the result  in  the case $p\leq 4$. In the case $p\geq 4$ we argue in the same manner by replacing    $\beta_HN^{2-p\over 2} \left( \omega^\prime(| x|)\right)^{p-2} \omega^{\prime \prime} (|x|)$
     by 
     the right hand side of (\ref{p>4}). 
    By  the conclusion (\ref{p>4}) in Proposition \ref{prop5}
      
       \begin{eqnarray*}
      && \lambda_1(M^{p-2} \Theta(x)(X+Y-2(2M+1) Id) \Theta(x)) \\
      &\leq& 4M^{p-1} {(1- N|x|^{2\epsilon}) \over \# I(x , \epsilon)}  \left(\omega^\prime(|x |)\right)^{p-2} |x |^{(p-4)\epsilon} 
        \left( { \omega^{\prime \prime} (|x |)\over 2} (1-N |x |^{2\epsilon})\right.\\
        &+& \left.  {3N\omega^\prime (|x | )\over 2}  |x |^{-1+2\epsilon}\right)\\
       &\leq &  M^{p-1}  \omega^{\prime \prime } (|x|)  \omega^\prime (|x|) ^{p-2} |x|^{(p-4) \epsilon}
        \end{eqnarray*}

\end{proof}
   
       \bigskip  
In the sequel we will use Proposition \ref{prop5} in the following context : 
For $x_o\in B_r$, let $M$ be a  constant to be defined later and  $\omega$ an increasing function which, near zero, behaves in the H\"older's case as $\omega(s) = s^\gamma$ and in the Lipschitz case like 
$\omega(s)=s$.  We define, borrowing ideas from \cite{IS}, \cite{BCI}, \cite{BD2}, 
 the function 
 \begin{equation}\label{psi}
     \psi(x,y) = u(x)-v(y)-\sup (u-v) - M \omega (|x-y|) -M |x-x_o|^2 -M |y-x_o|^2.
     \end{equation} 
If there exists $M$  independent on $x_o\in B_r$ such that $\psi(x,y)\leq 0$ in $B_1^2$ then the 
desired result holds. Indeed taking first $ x= x_o$ and using $|x_o-y| \leq 2$ one gets 
      $$ u(x_o)-v(y)\leq \sup (u-v)+ M(1+ 2^{2-\gamma})  \omega (|x_o-y|)$$
       and by taking secondly $y= x_o$ 
       $$ u(x)-v(x_o) \leq \sup (u-v)+ M(1+ 2^{2-\gamma})    \omega (|x_o-x|).$$
       
 In fact,  it is sufficient to prove the following : 
  
   There exists $\delta$ depending on $(r,p,  N)$  and $M$ depending on the same variables, such that for $|x-y|\leq \delta$
   $$ \psi(x,y)\leq 0.$$
Then  assuming in addition that 
$$M >1+  {4( |u|_\infty +  |v|_\infty) \over \omega (\delta)}$$ 
one gets $\psi(x,y)\leq 0$ anywhere in $B_1^2$.

  Suppose then  that 
  \begin{equation} \label{eqM}
  M(1-r)^2 > 8 (| u|_\infty + |v|_\infty), \ {\rm and} \  M >1+  {4( |u|_\infty +  |v|_\infty) \over \omega (\delta)}.
  \end{equation}      
Assume  by contradiction that the supremum of $\psi$ is positive. Then it is achieved on some $(\bar x, \bar y)$ which satisfy 
$|\bar x-x_o| , |\bar y-x_o| < {1-r\over 2}$. In particular $(\bar x,\bar y)\in  B_{1+r\over 2}^2$, hence  $\bar x$ and $\bar y$ are in the  interior of 
$B_1$, furthermore $|\bar  x -\bar y| < \delta$. Then by the consequences of Lemma \ref{lem2}  there exist
$$(q^x, X
 ) \in \overline{J} ^{2,+} u(\bar x)\ \mbox{ and}  \
     (q^y, -Y
      )\in \overline{J}^{2,-}v(\bar y)$$ with 
$q^x = M \omega^\prime (|\bar x-\bar y|) {\bar x-\bar y\over |\bar x-\bar y|} + 2M (\bar x-x_o)$
and     
     $q^y= M \omega^\prime (|\bar x-\bar y|) {\bar x-\bar y\over |\bar x-\bar y|} - 2M (\bar y-x_o)$.  
     Furthermore with the notations above, ie  $g(x) = \omega (|x|)$,  $H_1(\bar x-\bar y)= D^2 g(|\bar x-\bar y|) $, $(\Theta)_{ij}  =   \left\vert \omega^\prime (|\bar x-\bar y|) {\bar x_i-\bar y_i\over |\bar x-\bar y|}\right\vert ^{p-2\over 2}  \delta_i^j$ and  $ \tilde H(\bar x-\bar y)=(H_1(\bar x-\bar y)+ 2\iota H_1^2(\bar x-\bar y) ) $),  ($\iota = {1\over 4 M |H_1(\bar x-\bar y)|}$) 
     
        \begin{eqnarray}\label{ine}
      -   6M| H_1(\bar x-\bar y)|\left(\begin{array} {cc}
     I& 0\\
     0& I\end{array} \right) 
     &\leq& \left(\begin{array}{cc}
     X-(2M+1) Id  & 0\\
     0& Y-(2M+1) Id \end{array} \right) \nonumber \\
     &\leq&M \left( \begin{array}{cc}   \tilde H (\bar x-\bar y) & -   \tilde  H(\bar x-\bar y)\\
     - \tilde H(\bar x-\bar y)&   \tilde H
(\bar x-\bar y)     \end{array} \right) 
     \end{eqnarray} 
  Finally defining 
     $q = M \omega^\prime (|\bar x-\bar y|) {\bar x-\bar y\over |\bar x-\bar y|}$,   in each of the cases below  we will prove the following :       
            
{\bf Claims.} {\em There exists $\hat \tau >0$,  depending on $p$ (and on  $\gamma$ in the H\"older's case), and some constant $c>0$  depending on $(r,p, N)$  (and  on $\gamma$ in the H\"older's case) such that  for $\delta<1$  depending on $( r, p, N)$, and for   $|\bar x-\bar y| < \delta$     
the matrix $M^{p-2} \Theta(\bar x-\bar y) ( X+Y) \Theta(\bar x-\bar y)$ has one eigenvalue $\lambda_1$ such that
\begin{equation}\label{vpneg}
        \lambda_1(M^{p-2} \Theta(\bar x-\bar y) ( X+Y) \Theta(\bar x-\bar y)) \leq   -cM^{p-1}|\bar x-\bar y|^{-\hat \tau}
\end{equation}
$\exists c_1, \ {\rm and} \   \tau_1 < \tau,$ both depending \ on $(r,p,N)$  \ such \ that \ for\  all \   $i\geq 2$ 
  \begin{equation} \label{autresvp}
\ \lambda_i (M^{p-2} \Theta(\bar x-\bar y) (X+Y) \Theta(\bar x-\bar y)) \leq 
c_1 M^{p-1 }|\bar x-\bar y|^{-\tau_1}.
\end{equation}

There exist $\tau_2 <\hat  \tau$, and $c_2$  both depending  on  $(r,p,N) $ so that 
         
\begin{equation}\label{eqqx}
          ||q^x|^{p-2}  -|q|^{p-2}| |X| + || q^y|^{p-2}-|q|^{p-2} | |Y|\leq c_2 M^{p-1}|\bar x-\bar y|^{- \tau_2}.
\end{equation}}

 All these claims  imply that taking 
 $\delta$ small enough depending on $c, c_1, c_2$, there exists $c_3$ such that  
 $$F(q^x, X)-F(q^y, -Y) \leq -c_3 M^{p-1} |\bar x-\bar y|^{-\hat \tau}.$$

 Indeed 
 \begin{eqnarray*}
   F(q^x, X) &=& \sum_{i=1}^{i=N}  |q_i^x|^{p-2} X_{ii} \\
   &\leq& \sum_1^N |q_i|^{p-2} X_{ii} +    ||q^x|^{p-2}  -|q|^{p-2}| |X| \\
   &\leq & \sum_1^N |q_i|^{p-2} (-Y)_{ii} + M^{p-2}  \sum_1^N \lambda_i (\Theta(\bar x-\bar y) (X+Y) \Theta(\bar x-\bar y) ) \\
   & +&     ||q^x|^{p-2}  -|q|^{p-2}| |X|\\
   &\leq &   \sum_i |q_i|^{p-2} (-Y)_{ii}+ M^{p-2} 
\lambda_1(\Theta(\bar x-\bar y) (X+Y) \Theta(\bar x-\bar y)) \\
&+& M^{p-2}  \sum_{i\geq 2}  \lambda_i (\Theta(\bar x-\bar y) (X+Y) \Theta(\bar x-\bar y) ) 
+  ||q^x|^{p-2}  -|q|^{p-2}| |X|\\
 &\leq &    \sum_i |q_i^y|^{p-2} (-Y)_{ii} -c M^{p-1}|\bar x-\bar y|^{-\hat \tau}
+ (N-1)  c_1 M^{p-1}|\bar x-\bar y|^{- \tau_1} \\
&+&|| q^y|^{p-2}-|q|^{p-2} | |Y|+   ||q^x|^{p-2}  -|q|^{p-2}| |X|\\
&\leq & F(q^y, -Y)  -c M^{p-1}|\bar x-\bar y|^{-\hat \tau}
+ (N-1)  c_1 M^{p-1}|\bar x-\bar y|^{- \tau_1} + c_2 M^{p-1}|\bar x-\bar y|^{- \tau_2}\\
&\leq& F(q^y, -Y) -c_3 M^{p-1} |\bar x-\bar y|^{-\hat \tau}.
\end{eqnarray*}
as soon as 
$c_1(N-1) \delta^{\hat \tau-\tau_1} + c_2 \delta ^{\hat \tau-\tau_2} < {c\over 2}$.

 Then one can conclude using the alternative Definition \ref{altdef}  of viscosity sub- and super-solutions,  since 
 
 \begin{eqnarray*}
  f(\bar x) &\leq & F(q^x, X) \\&\leq &F(q^y, -Y) -c_3 M^{p-1} |\bar x-\bar y|^{-\hat \tau}\\
  &\leq& g(\bar y) -c_3 M^{p-1} |\bar x-\bar y|^{-\hat \tau}.
  \end{eqnarray*}
  This is clearly false as soon as $\delta$ is small enough  since $f$ and $g$ are bounded. 
  
   So  in order to get the H\"older and Lipschitz regularity in the case $p\leq 4$ and $p\geq 4$ it is sufficient to prove (\ref{vpneg}), (\ref{autresvp}) and (\ref{eqqx}) in any cases, 
   note that the cases $p\geq 4$ will also  require to check    (\ref{eqNepsilon}).

      Note that (\ref{eqqx}) will always be a consequence of (\ref{ZT}) below and of (\ref{majnorm}): 
 
  For any $\theta \in ]0, \inf (1,(p-2)]$, and for any $Z, \ T\in \R^N$,     
    \begin{equation}\label{ZT}
       ||Z|^{p-2} -|T|^{p-2} | \leq \sup (1,p-2)  |Z-T|^{\theta }(|Z| + |T|)^{p-2-\theta} .
       \end{equation}
       
       This is obtained   for $p-2\leq 1$,  from   the inequality $||Z|^{p-2}-|T|^{p-2} | \leq |Z-T|^{p-2}$,  
       and  for  $p\geq 3$, using the  mean values 's Theorem.

    \subsection{Proofs of (\ref{vpneg}), (\ref{autresvp}) and (\ref{eqqx}) in the H\"older's  case and $p\leq 4$}

  Let  $r<1, $ $\gamma$ be some number  in $ ]0,1[$, and $\omega (s) = s^\gamma$.  
   $\psi$ is defined by (\ref{psi}) and $M$ will be chosen large later independently  on $x_o$, but depending on $r$ , $\gamma$,  $p$,  $N$, $|f|_\infty$, $|g|_\infty$, $|u|_\infty$ and $|v|_\infty$ . 
    
    Note that  below  even if we did not  always make it explicit for simplicity,  some of the  constants depend on ${1\over 1-\gamma}$  and then  the Lipschitz result   cannot be derived immediately from the H\"older's one by letting $\gamma$ go to $1$. 
    
    Let us recall that 
     we want to prove that  there exists $\delta$ depending on $(r,p, N)$  and $M$ depending on the same variables, and on $|f|_\infty$, $|g|_\infty$, $|u|_\infty$ and $|v|_\infty$,  such that for $|x-y|\leq \delta$
   $$ \psi(x,y)\leq 0.$$
Recall that we impose  (\ref{eqM}), ie   $  M (1-r)^2 > 8(|u|_\infty +|v|_\infty),\ {\rm and } \ M > 1+ { 4(|u|_\infty +  |v|_\infty) \over \omega (\delta)}.$

      Let $(\bar x, \bar y)$ be  some couple   in $(\overline{B_1})^2$ on which the supremum is positive and achieved.  Clearly $\bar x\neq \bar y$,     
    and from the assumptions on $M$, $(\bar x, \bar y)\in B_{1+r\over 2} \times B_{1+r\over 2}$, and   $ |\bar x-\bar y| <\delta$. 
 Here   
 
   $$q^x =\gamma M  |\bar x-\bar y|^{\gamma-2} (\bar x-\bar y) + 2M ( \bar x-x_o),\ q^y =\gamma  M|\bar x-\bar y|^{\gamma-2} (\bar x-\bar y) -2M  (\bar y-x_o). $$
 and 
  $$q=M \omega^\prime (|\bar x-\bar y|) {\bar x-\bar y\over |\bar x-\bar y|} = \gamma M  |\bar x-\bar y|^{\gamma- 2} (\bar x-\bar y). $$ 
      In particular
      $$ |q| = M\gamma |\bar x-\bar y|^{\gamma-1}$$
       and for $\delta^{1-\gamma} < {\gamma\over 8}$ one has 
       $$ {|q|\over 2} \leq |q^x|, |q^y|  \leq {3|q|\over 2}.$$
       Note that $\Theta $ defined in the previous sub-section is also  the matrix 
   $$ \Theta_{ij}  (\bar x-\bar y) = \left\vert{q_i\over M} \right\vert ^{p-2\over 2} \delta_i^j     $$

  Furthermore $|H_1(\bar x-\bar y)| = |D^2 g(|\bar x-\bar y|)| \leq |\omega^{\prime \prime } (|\bar x-\bar y|) |+ (N-1) 
  {\omega^\prime (\bar x-\bar y)\over |\bar x-\bar y|}= \gamma (\gamma+ N-2) |\bar x-\bar y|^{\gamma-2}$,
   and  so by (\ref{majnorm})   and for $\delta$ small enough  
   
    \begin{equation}\label{majnormhold}|X|+ |Y| \leq c M |\bar x-\bar y|^{\gamma-2}.
    \end{equation}  
    
    Using  (\ref{mu1pleq4}) one has 
    $ \lambda_1 (M^{p-2} \Theta(\bar x-\bar y) (X+Y-2(2M+1)  Id )\Theta(\bar x-\bar y)) \leq c(M \gamma)^{p-1} (\gamma-1) |\bar x-\bar y|^{(\gamma-1)(p-2) + \gamma-2}$, 
     Hence  
     \begin{eqnarray*}
     \lambda_1(M^{p-2} \Theta(\bar x-\bar y) (X+Y) \Theta(\bar x-\bar y)) &\leq& -cM^{p-1} (1-\gamma) \gamma^{p-1} |\bar x-\bar y|^{(\gamma-1)(p-2) + \gamma-2}\\
     &+ & 2(2M+1) M^{p-2}  |\Theta(\bar x-\bar y) |^2\\
     & \leq & -{c\over 2}M^{p-1}  (1-\gamma) \gamma^{p-1} |\bar x-\bar y|^{(\gamma-1)(p-2) + \gamma-2} 
     \end{eqnarray*}
      as soon as $\delta^{2-\gamma} < {c(1-\gamma) \gamma^{p-1}\over 12}  $. Hence    
     (\ref{vpneg}) holds with $\hat \tau = (1-\gamma) (p-2) + 2-\gamma$.  
     
     \medskip

    On the other hand  by (\ref{autresvp1})
     $\lambda_i (M^{p-2} \Theta(\bar x-\bar y) (X+Y) \Theta(\bar x-\bar y)) \leq 2(2M+1) M^{p-2}   |\Theta |^2((\bar x-\bar y)) \leq 2(2M+1)  M ^{p-2}  \gamma^{p-2} |\bar x-\bar y|^{(\gamma-1)(p-2)}$
      and then (\ref{autresvp})  holds with 
      $\tau_1 = (1-\gamma) (p-2)$.

         \medskip

       To check (\ref{eqqx}) let us use (\ref{ZT}) and (\ref{majnormhold}) to get 
       $$||q^x|^{p-2} -|q|^{p-2}| |X| \leq c_2M^{p-1}  |\bar x-\bar y|^{(\gamma-1)(p-2-\theta) + \gamma-2}$$
        where $\theta  = \inf (1, p-2)>0$ and then 
         (\ref{eqqx}) holds with 
         $\tau_2 = (1-\gamma) (p-3)^+ + 2-\gamma$.

       \subsection{Proof of (\ref{eqNepsilon}), (\ref{vpneg}), (\ref{autresvp}), (\ref{eqqx} ) in the H\"older's case and $p\geq 4$}   
       
        $\omega$ is the same as in the     H\"older's case and $p\leq 4$.   
   We define 
       \begin{equation}\label{eq21}
       \epsilon = {(1-\gamma)\over 2(p-4)},\ \delta_N = \exp ({-\log (2N(4-\gamma))+ \log (1-\gamma) \over 2\epsilon} ),       \end{equation}
          and assume $\delta < \delta_N$. One still  suppose (\ref{eqM}).
          
In    particular    since there exists $i\in [1,N]$ such that $|\bar x_i-\bar y_i|^2 \geq{ |\bar x-\bar y|^2\over N}\geq |\bar x-\bar y|^{2+ 2\epsilon} $, for $p\geq 4$, using the definition of $\delta_N$ in  (\ref{eq21}),  $I(\bar x-\bar y,\epsilon)\neq \emptyset$. 
           Furthermore  for $|\bar x-\bar y| < \delta \leq \delta_N$
             \begin{equation}\label{eq22} 
            \sum_{i\in I} |\bar x_i-\bar y_i|^2 \geq |\bar x-\bar y|^2  (1-{1-\gamma\over 2N(4-\gamma)})\geq {3\over 4} |\bar x-\bar y|^2
            \end{equation}
            and
              \begin{eqnarray}
              \label{eq20} 
          {1\over 2}   \omega^{\prime \prime } (|\bar x-\bar y|)(1-N |\bar x-\bar y|^{2\epsilon} ) &&+ {3N\over 2}  |\bar x-\bar y|^{2\epsilon} {\omega^\prime (|\bar x-\bar y|)\over |\bar x -\bar y|}\nonumber\\
                     & \leq &{1\over 2}   \omega^{\prime \prime } (|\bar x-\bar y|) + |\bar x-\bar y|^{2\epsilon} ( {N\over 2} \gamma (1-\gamma) + {3N\over 2} \gamma) |\bar x-\bar y|^{\gamma-2}\nonumber \\
                     &\leq &  {1\over 4} \gamma (\gamma-1) (|\bar x-\bar y|)^{\gamma-2}  \nonumber\\
           &\leq &{\omega^{\prime \prime} (|\bar x-\bar y)\over 4} ,
            \end{eqnarray} 
   and then (\ref{eqNepsilon}) is satisfied.  Using (\ref{eq20}), (\ref{eq22}),  and (\ref{mu1pgeq4}) one gets 
   $$ \lambda_1(M^{p-2} \Theta(\bar x-\bar y) (X+Y-2(2M+1) Id) \Theta(\bar x-\bar y) ) \leq cM^{p-1}  \omega^{\prime \prime } (|\bar x-\bar y|) \omega^{\prime } (|\bar x-\bar y|)^{p-2}   |\bar x-\bar y|^{(p-4) \epsilon}$$
   and then 
    \begin{eqnarray*}
     \lambda_1(M^{p-2} \Theta(\bar x-\bar y) (X+Y) \Theta(\bar x-\bar y)) &\leq& -cM^{p-1} (1-\gamma) \gamma^{p-1} |\bar x-\bar y|^{(\gamma-1)(p-2) + \gamma-2+ (p-4) \epsilon}\\
     &+ & 2(2M+1) M^{p-2}  |\Theta(\bar x-\bar y) |^2\\
     &\leq & -cM^{p-1} (1-\gamma) \gamma^{p-1} |\bar x-\bar y|^{(\gamma-1)(p-2) + \gamma-2+ (p-4) \epsilon}\\
     &+ &(2(2M+1) M^{p-2} |\bar x-\bar y|^{(\gamma-1)(p-2)}\\
     & \leq & -{c\over 2}M^{p-1}  (1-\gamma) \gamma^{p-1} |\bar x-\bar y|^{(\gamma-1)(p-2) + \gamma-2+ (p-4) \epsilon} 
     \end{eqnarray*}
      as soon as $\delta$  is small enough, by the choice  of $\epsilon$ in (\ref{eq21}). 
   
   Then (\ref{vpneg}) holds with 
$\hat \tau = (1-\gamma) (p-2) + 2-\gamma -(p-4)\epsilon>0$, by  (\ref{eq21}) while  (\ref{autresvp}) holds  with $\tau_1 = (1-\gamma) (p-2)$ since 
$M^{p-1}  |\Theta(\bar x-\bar y) |^2 \leq c_1 M^{p-1} |\bar x-\bar y|^{(\gamma-1) (p-2)} $
and   $(1-\gamma)(p-2) <  (1-\gamma)(p-2) + (2-\gamma)-(p-4) \epsilon$. 
 
   Finally using 
   (\ref{ZT}) and (\ref{majnormhold}) 
                 
                 \begin{eqnarray*}
               ||q_i^ x|^{p-2} -|q_i|^{p-2} | |X_{ii} |& \leq&  (p-2)|q_i^x-q_i| (|q_i ^x | + |q_i|)^{p-3}c_1M |\bar x-\bar y|^{\gamma-2} \\\
               &\leq &
             c_2  M^{p-1} |\bar x-\bar y|^{(\gamma-1)(p-3)+ \gamma-2 }
             \end{eqnarray*} 
              and then 
               (\ref{eqqx}) is satisfied with  $\tau_2 =    (1-\gamma)(p-3)+2- \gamma <\hat \tau$ by (\ref{eq21}).

      \begin{rema} 
       Suppose that $u= v$ and $f= g$.                                         
     From the previous proof one gets that  for any $\gamma \in ]0,1[$,  there exists some constant $C_{p,\gamma,  N, r}$ such that if $u$ is a solution of (\ref{eq1}) in $B_1$,  such that $|u|_\infty \leq 1$ and $|f|_\infty \leq 1$,
     $ |u|_{{\cal C}^{0, \gamma}(B_r)}\leq C_{p,\gamma, N, r}$. 
     Let now  $u$  be a  bounded  solution in the ball $B_1$,  then $v = {u\over |u|_\infty + |f|_\infty^{1\over p-1}}$ satisfies the equation, with  $|v|_\infty \leq 1$ and the right hand side $\tilde f  = {f\over ( |u|_\infty + |f|_\infty^{1\over p-1})^{p-1}}$.   Then 
      $|u|_{{\cal C}^{0, \gamma}(B_r)}\leq C_{p, \gamma, N , r} (|u|_\infty + |f|_\infty^{1\over p-1})$.
\end{rema}

\subsection{Proof of (\ref{vpneg}), (\ref{autresvp}) and (\ref{eqqx}) in the Lipschitz case and $p\leq 4$}

                We choose $\tau \in ]0,  \inf ({1\over 2},{ p-2\over 2})[$,  $\gamma> {\tau\over \inf ({1\over 2},{ p-2\over 2})}$,  $\gamma<1$,  and   define 
                $ \omega(s) = s-\omega_o s^{1+\tau}$
                 where $s < s_o=\left( {1\over (1+\tau) \omega_o}\right)^{1\over \tau}$ and $\omega_o$ is chosen so that  $s_o>1$. 
                 We suppose that $\delta^\tau \omega_o (1+\tau) < {1\over 2}$, which ensures that 
                 \begin{equation}\label{omegaprime}
                 {1\over 2} \leq \omega^\prime (s) < 1, \ {\rm for} \ s < \delta
                 \end{equation}
                  and by the mean value's theorem, for $s < \delta$, 
 $\omega(s) \geq {s\over 2} $. Note that $\omega$ is globally Lipschitz continuous. We recall that  by (\ref{eqM}) we choose  $M(1-r)^2 > 8 (|u|_\infty+ |v|_\infty)$  and  
$ M >1+  {8(|u|_\infty|+ |v|_\infty) \over \delta}  . $
 
 Suppose that $(\bar x, \bar y)$ is a  pair on which the supremum of $\psi$ is achieved $>0$. As in the previous subsections, $(\bar x, \bar y)\in B_{1+r\over 2}^2$ and 
 $|\bar x-\bar y| < \delta$. 
  Here  $q^x = M \omega^\prime (|\bar x-\bar y|) {\bar x-\bar y\over |\bar x-\bar y|}+ 2M (\bar x-x_o), \            
                  q^y = M \omega^\prime (|\bar x-\bar y|){\bar x-\bar y\over |\bar x-\bar y|} + 2M (\bar y-x_o)$, and we also define 
                  $ q= M\omega^\prime (|\bar x-\bar y|){\bar x-\bar y\over |\bar x-\bar y|} $.  In particular ${M\over 2} \leq |q| \leq M$. 
       Note that 
since the solution has been proven to be H\"older  in $B_{1+r\over 2}$ for 
all $\gamma <1$, for some constant $c_{p, \gamma, N, r}$, 
  from $\psi(\bar x, \bar y) \leq 0$ and $\omega(|\bar x-\bar y|) \geq 0$ one gets 
$M|\bar x-x_o|^2+M[\bar y-x_o|^2 \leq c_{p, \gamma,N, r}  |\bar x-\bar y|^\gamma$
and  then  

\begin{equation}\label{x-xo}
 |\bar x-x_o|\leq \left({c_{p, \gamma,N, r} |\bar x-\bar y|^\gamma\over M}\right)^{1\over 2}.
 \end{equation} 
  and an analogous estimate holds for $|\bar y-x_o|$,  then  taking 
 $\delta$ small enough,  by (\ref{omegaprime}), 
                                   
                                   \begin{equation}\label{ineqx}
                                   {M\over 4} \leq |q^x|, |q^y| \leq {5M\over 4}
                                   \end{equation}
                                   
                                    To prove (\ref{vpneg}) let us observe that   here one has 
                                    $$(\omega^\prime (|\bar x-\bar y|) )^{p-2} \omega^{\prime \prime } (|\bar x-\bar y|) \leq -c |\bar x-\bar y|^{\tau-1}$$ and using 
                                    $M^{p-2} |\Theta|^2 \leq M^{p-2}$,  then by (\ref{mu1pleq4}) ,   for $\delta$ small enough,  
                                   (\ref{vpneg}) holds with $\hat \tau = 1-\tau$, while (\ref{autresvp}) holds with $\tau_1 = 0$.

                                   \medskip
                                   
                                    To check (\ref{eqqx}), 
                                    let us observe that here 
                                    $ \vert D^2 g(|\bar x-\bar y|)\vert \leq  |\omega^{\prime \prime} (|\bar x-\bar y|)|+  (N-1) {\omega^\prime (|\bar x-\bar y|) \over |\bar x-\bar y|}\leq  {c\over |\bar x-\bar y|}$ and then using 
                                   ( \ref{ine})
                                   
                                     $$ |X-(2M+1)\ Id| + |Y-(2M+1) \ Id | \leq c{M\over |\bar x-\bar y|}$$
                                   
                                      Hence for $\delta$ small enough one also has 
                                       \begin{equation}\label{inelip} |X| + |Y| \leq c{M\over |\bar x-\bar y|}   \end{equation}
                                    
                                   and then by (\ref{ZT}),  (\ref{ineqx}),   (\ref{x-xo}),  and  (\ref{inelip})
            $$ |\ |q^x|^{p-2}-|q |^{p-2} | \ |X| \leq   c M^{p-1}c_{p, \gamma,N, r}^{\sup (1, p-2) \over 2}    |\bar x-\bar y| ^{\inf (1, p-2)  \gamma \over 2}  |\bar x-\bar y|^{-1}$$ 
hence 
 (\ref{eqqx}) holds with 
 $\tau_2 = 1- {\inf (1, p-2) \gamma\over 2} < 1-\tau$ by the choice of $\gamma$.

            \subsection { Proofs of (\ref{eqNepsilon}),  (\ref{vpneg}), (\ref{autresvp}), (\ref{eqqx}) in the Lipschitz case and $p \geq 4$}
             
             We take the same function $\psi$ and $\omega$ as in the Lipschitz case and $p\leq 4$, with   still 
             $M(1-r)^2 > 8 (|u|_\infty+ |v|_\infty)$  and  
$
 M > 1+{8(|u|_\infty|+ |v|_\infty )\over \delta}  . 
$
              We choose   also 
              \begin{equation}\label{epsilonlip}
              0< \tau  <  {1\over (p-2)}, \   1>\gamma > \tau(p-2), \ {\rm and }  \ 
              {\tau\over 2} < \epsilon<{{\gamma\over 2}-\tau\over ( p-4)} . 
              \end{equation} 
              (Recall that it has been proven  in Subsection 3.2 that $u$ is $\gamma$- H\"older continuous   in $B_{1+r\over 2}$). 
        Let us define $\omega$, $s_o$, $g$ $\psi$, as in the case $p\leq 4$,  $\Theta$, $\tilde H$ and $H$ 
         are as in the Subsection 3.1 .             
            We  define  
            \begin{eqnarray}\label{deltaN}
            \delta_N &=&
            \inf \left(    \exp ({\log (\omega_o (1+\tau) \tau)-\log (2N(\omega_o \tau (1+ \tau)+3))\over 2\epsilon-\tau},\right.\nonumber\\
            &&\left.  \exp\left( {-\log (2 \omega_o (1+ \tau))\over \tau}\right)\right). 
                      \end{eqnarray}
                      
                       Note that $\delta_N \leq \exp\left( { -\log (2N)\over 2\epsilon}\right)$. 
                 If we suppose  that $\sup_{(x,y)\in B_1^2} \psi(x, y)>0$,  it is achieved on some $(\bar x, \bar y) $ which satisfies 
             $\bar x\neq \bar y$,  
            $(\bar x, \bar y)\in B_{1+r\over 2} \times B_{1+r\over 2} $ and 
            $|\bar x-\bar y| \leq \delta_N$. In particular  since there exists $i$ such that 
            $|\bar x_i-\bar y_i|^2 \geq {1\over N} |\bar x-\bar y|^2 \geq|\bar x-\bar y|^{2+ 2\epsilon} $, by (\ref{deltaN}),  $I(\bar x-\bar y, \epsilon)\neq \emptyset$, and 
            \begin{eqnarray*}
            \sum_{i\in I(\bar x-\bar y, \epsilon) }|\bar x_i-\bar y_i|^2&=&|\bar x-\bar y|^2-\sum_{i\notin I(\bar x-\bar y, \epsilon)} |\bar x_i-\bar y_i|^{2}\\
            &\geq& |\bar x-\bar y|^2-N |\bar x-\bar y|^{2+ 2\epsilon}\\
            & \geq& {1\over 2} |\bar x-\bar y|^2 .
            \end{eqnarray*}
             Furthermore recall that  by (\ref{deltaN}), $1\geq \omega^\prime (|\bar x-\bar y|) \geq {1\over 2}$ and  
             \begin{eqnarray*}
              {1\over 2} \omega^{\prime \prime } (|\bar x-\bar y|) &+&
                {N\over 2}  \omega_o \tau (1+ \tau) |\bar x-\bar y|^{\tau-1+ 2\epsilon} +  {3\over 2} N |\bar x-\bar y|^{2\epsilon-1} \omega^\prime (|\bar x-\bar y|) \\
                &\leq & {1\over 2} \omega^{\prime \prime } (|\bar x-\bar y|)+ {N\over 2} ( \omega_o \tau (1+ \tau)  +3) |\bar x-\bar y|^{2\epsilon-1} \\
                &\leq& 
              -{1\over 4} \omega_o (1+\tau) \tau |\bar x-\bar y|^{-1+ \tau } \\
           & = & { \omega^{\prime \prime } (|\bar x-\bar y|)\over 4}. 
            \end{eqnarray*}
             and then (\ref{eqNepsilon}) holds.  This implies that the right hand side of (\ref{mu1pgeq4}) is for $x:= \bar x-\bar  y$ less than 
             $$cM^{p-1}  \omega^{\prime} (|\bar x-\bar y|)^{p-2} \omega^{\prime \prime } (|\bar x-\bar y|) |\bar x-\bar y|^{p-4}\leq -c M^{p-1} |\bar x-\bar y|^{-\hat \tau}$$
             
                                where by (\ref{epsilonlip}) 
                   $\tau-1+ (p-4) \epsilon:= -\hat \tau <0$, 
                                                      and using $4M^{p-1} |\Theta(\bar x-\bar y) |^2 \leq c M^{p-1}$ one gets (\ref{vpneg}) for $\delta$ small enough.

                                                   Always by   (\ref{autresvp1})  and since 
         $M^{p-1}   |\Theta(\bar x-\bar y) \  {\rm Id}  \  \Theta(\bar x-\bar y)|\leq    M^{p-1}$,     (\ref{autresvp}) is satisfied with $\tau_1=0$.
                         
                   There remains to prove (\ref{eqqx}).  
  For that aim  observe that 
 there exists $c_{p, \gamma,N,  r}$ so that $M |\bar x-x_o|^2 \leq c_{p, \gamma, N, r} |\bar x-\bar y|^ {\gamma}$. 
 Therefore,  using (\ref{ZT}), and (\ref{inelip}), one gets 
 \begin{eqnarray*}
  |\ |q^x |^{p-2} -|q|^{p-2}|\ |X |&\leq&(p-2)  |q^x -q | (|q^x|+ |q|)^{p-3} |X|\\
 &\leq& c_2   M^{p-1} |\bar x-\bar y|^{{\gamma\over 2}-1}\\
 \end{eqnarray*}
  Hence (\ref{eqqx}) holds with $\tau_2 =1-{ \gamma\over 2} < 1-\tau-(p-4) \epsilon$ by  (\ref{epsilonlip}).

     {\bf Conclusion} 
     
     From the previous proof one gets  in the particular case where $u=v$ and $f=g$ that  there exists some constant $C_{p, N, r}$ such that if $u$ is a solution in $B_1$,  such that $|u|_\infty \leq 1$ and $|f|_\infty \leq 1$,
     ${\rm lip }_{B_r} u\leq C_{p,N, r}$. 
     Let now  $u$  be a  bounded  solution in the ball $B_1$,  then $v = {u\over |u|_\infty + |f|_\infty^{1\over p-1}}$ satisfies the equation, with  $|v|_\infty \leq 1$ and the right hand side $\tilde f  = {f\over ( |u|_\infty + |f|_\infty^{1\over p-1})^{p-1}}$.   Then 
      ${\rm lip}_{B_r}  u\leq C_{p, N , r} (|u|_\infty + |f|_\infty^{1\over p-1})$.

\bigskip

 {\em Acknowledgment : The author wishes to thank Guillaume Carlier and Lorenzo Brasco for fruitful discussions on this subject,  and the anonymous referees  for their judicious remarks which permit to considerably improve this  paper. }


\begin{thebibliography}{99}

 \bibitem{BCI} G. Barles,  E. Chasseigne, C. Imbert, {\it H\"older continuity of solutions of second-order non-linear elliptic integro-differential equations},  J. Eur. Math. Soc. 
 vol {13} (2011),  p 1-26.

\bibitem{BK} M. Belloni B. Kawohl  {\em The  Pseudo p-Laplace  eigenvalue problem and viscosity solutions},  
ESAIM COCV, vol. 10, (2004), p 28-52. 

\bibitem{BD3} I. Birindelli, F. Demengel, {\em  ${\cal C}^{1,\beta}$ regularity for Dirichlet problems associated to fully nonlinear degenerate elliptic 
equations},  ESAIM  COCV, vol  20, Issue 40, (2014), p. 1009-1024.  et  https://hal.archives-ouvertes.fr.hal-01076713.


\bibitem{BD2} I. Birindelli, F. Demengel, {\em  Existence and regularity result for sub and super-solutions   of pseudo-Pucci's operators},   work in  preparation.

 \bibitem{BB}
 P. Bousquet, L. Brasco {\em  Global  Lipschitz Continuity
 for Minima of degenerate  Problems},   Arxiv : 1504 06 101

\bibitem{BBJ}
P. Bousquet,   L. Brasco, V. Julin, {\em Lipschitz regularity for local minimizers of some widely degenerate problems},  Calc. of Variations and Geometric Measure theory. http://cvgmt.sns.it/paper/2515/
 
 
\bibitem{BC} L.  Brasco, G. Carlier, {\em
   On certain anisotropic elliptic equations arising in congestion optimal transport : Local gradient bounds }, 
  Advances in Calculus of Variations, Vol 7,  (2014), Issue 3,p 379-407. 




\bibitem{CC} X. Cabr\'e, L. Caffarelli,   Fully-nonlinear equations, 
Colloquium Publications 43, American Mathematical Society, Providence, RI,1995.



\bibitem{CF} M. Colombo, A. Figalli,  {\em Regularity Results  for  very  degenerate  elliptic  equations}, J. Math. Pures Appl. (9), 101 (2014), no. 1, 94-117.

\bibitem{usr} { M.G. Crandall,  H. Ishii, P.L. Lions},  {\em User's guide to
viscosity solutions of second order partial differential equations}, Bull.
Amer. Math. Soc. (N.S.) 27 (1992), no. 1, p 1-67.
\bibitem{CKLS} {M. Crandall, Kocan, P.L. Lions, Swiech},  {\em On viscosity solutions of Fully Nonlinear Equations with Measurable Ingredients. 
} Communications on pure and applied Mathematics, Vol. XLIX, (1996), p 365-397 . 
\bibitem{DB} E. Di Benedetto,  {\em ${\cal C}^{1+\beta}$  local regularity of weak solutions of degenerate elliptic equations,} Nonlinear  Analysis, Theory,  Methods  and  Applications, Vol. 7. No. 8. pp. 827-850, 1983.


\bibitem{FFM} I. Fonseca, N. Fusco,  P. Marcellini, 
{\em An existence result for a non convex variational problem via regularity}, ESAIM: Control, Optimisation and Calculus of Variations, Vol. 7, ( 2002),  p 69-95. 
\bibitem{G} E. Giusti,  Direct methods in the calculus of variations. World Scientific Publishing Co., Inc., River Edge, NJ, 2003. 

\bibitem{IS} C. Imbert, L. Silvestre,  {\em ${\cal C}^{1, \alpha}$   regularity of solutions of degenerate fully non-linear elliptic equations} , Adv. Math. vol 233, (2013), p 196-206.

\bibitem{I1} {H. Ishii},  {\it Viscosity solutions  of Nonlinear Partial Differential equations} Sugaku Expositions , Vol 9, number 2, December 1996. 
\bibitem{IL}{H. Ishii, P.L. Lions}, {\it  Viscosity solutions of Fully-Nonlinear Second  Order Elliptic Partial Differential Equations}, {J. Differential Equations},  vol 83,   (1990),  p 26-78.


\bibitem{T} P. Tolksdorff , {\em Regularity for a more general class of quasilinear elliptic
equations}, J. Differential Equations, 51 (1984), 126-150.


 \bibitem{UU}N.  Uraltseva, N.  Urdaletova,  {\it  The boundedness of the gradients of generalized solutions of degenerate quasilinear nonuniformly elliptic equations}, Vest. Leningr. Univ. Math., 16 (1984), 263-270.
 
          \end{thebibliography}
             \end{document}